\documentclass{article}
\usepackage{mathtools,amsthm, amssymb,bm}
\usepackage{enumerate,comment}
\usepackage{amsmath}
\usepackage{thm-restate}
\usepackage{graphicx}
\usepackage{nicematrix}
\usepackage{tikz}
\usepackage[normalem]{ulem}
\usetikzlibrary{positioning}

\theoremstyle{definition}
\newtheorem{Def}{Definition}[section]

\newtheorem{Rem}[Def]{Remark}

\newtheorem{Notation}[Def]{Notation}

\theoremstyle{plain}
\newtheorem{Thm}[Def]{Theorem}

\newtheorem{Prop}[Def]{Proposition}
\newtheorem{Lem}[Def]{Lemma}
\newtheorem{Cor}[Def]{Corollary}
\newtheorem{Fact}[Def]{Fact}

\newcommand{\an}[2][n]{\left(#2\mathrel;#1\right)}
\newcommand{\pd}[1]{\dfrac{\partial}{\partial #1}}
\newcommand{\genpdd}[2]{\dfrac{\partial^2}{\partial{#1}\partial{#2}}}
\newcommand{\pdd}[1]{\dfrac{\partial^2}{\partial{#1}^2}}

\newcommand{\fps}{[\![z_1,\ldots,z_{2g-1}]\!]}
\newcommand{\diff}[1]{\mathop{\mathcal{D}_{#1}}}
\newcommand{\diffchar}[2]{\mathop{\mathcal{D}_{#1}^{(#2)}}}
\newcommand{\diffsimp}[2]{\mathop{\mathcal{D}_{#1,\,#2}}}
\newcommand{\ALseries}{\mathcal{F}(a,b_1,\ldots,b_{2g-1},c\,;z_1,\ldots,z_{2g-1})}
\newcommand{\subALseriesdef}{%
\frac{\an[\displaystyle\sum_{k=1}^{2g-1} n_k]{a} \displaystyle\prod_{k=1}^{2g-1} \an[n_k]{b_k}}%
{\an[\displaystyle\sum_{k=1}^{2g-1} n_k]{c} \displaystyle\prod_{k=1}^{2g-1} \an[n_k]{1}}%
}
\newcommand{\ALseriesdef}{%
\sum_{(n_1,\ldots,n_{2g-1})\in \mathbb{N}_0^{2g-1}}\subALseriesdef%
}
\newcommand{\zzz}{z^{n_1}_1\cdots z^{n_{2g-1}}_{2g-1}}
\newcommand{\setsymbol}[2]{\left\{ #1 \mathrel{} \middle| \mathrel{} #2 \right\}}
\newcommand{\setsymbolin}[3]{\left\{ #1 \in #2 \mathrel{} \middle| \mathrel{} #3 \right\}}
\newcommand{\quotientset}[2][\sim]{\left.#2 \middle/ \mathord#1 \right.}
\newcommand{\nnn}{n_1,\ldots,n_{2g-1}}

\usepackage{xparse}
\NewDocumentCommand{\truncatedAL}{o o}{%
    \IfValueTF{#1}{%
        \IfValueT{#2}{%
            \widetilde{\mathcal{F}}_{#1,#2}
        }
    }{
        \widetilde{\mathcal{F}}_{i,j}
    }
}
\NewDocumentCommand{\supp}{o o}{%
    \IfValueTF{#1}{%
        \IfValueT{#2}{%
            \mathsf{supp}^{(#1,#2)}
        }
    }{
        \mathsf{supp}^{(i,j)}
    }
}
\NewDocumentCommand{\Legendresymbol}{o o}{%
    \IfValueTF{#1}{%
        \IfValueT{#2}{%
            \genfrac{(}{)}{}{}{#1}{#2}
        }
    }{
            \genfrac{(}{)}{}{}{-1}{p}
    }
}
\begin{document}
\title
{\bf The multiplicity-one theorem for the superspeciality of hyperelliptic curves}
\author{Yuya \textsc{Yamamoto}\thanks{Graduate School of Environment and Information Sciences, Yokohama National University.
E-mail: \texttt{yamamoto-yuya-cm@ynu.jp}}
}
\date{}
\maketitle
\begin{abstract}
\noindent\quad
The multiplicity-one theorem for the simultaneous equations characterizing the superspeciality of hyperelliptic curves was established by Igusa in 1958 for genus one, and later extended by Harashita and Yamamoto in 2026 to genus two. In this paper, we generalize this result to arbitrary genus. Our approach employs the Lauricella system of type D for hypergeometric series in $2g-1$ variables, whose truncations (up to scalar multiplication) give the entries of a Cartier--Manin matrix. The multiplicity-one theorem is obtained through an analysis of equalities involving partial derivatives of these entries.


\end{abstract}

\section{Introduction}
In this paper, we investigate the multiplicity of superspeciality for an entire family of hyperelliptic curves of arbitrary genus.
We begin by presenting the framework of our study, followed by our main theorems.

Let $p$ be an odd prime
and $K$ an algebraically closed field of characteristic $p$.
Let $C$ be the normalization of the projective model of
\[
\quad y^2 = f(x)\coloneqq x(x-1)(x-\lambda_1)(x-\lambda_2)\cdots(x-\lambda_{2g-1})
\]
for $\lambda_1,\ldots,\lambda_{2g-1} \in K$ with $\#\{0,1,\lambda_1,\ldots,\lambda_{2g-1}\} = 2g+1$.
Note that $C$ is a hyperelliptic curve of genus $g$ over $K$ and conversely  any hyperelliptic curve of genus $g$ over $K$ is written in this way.
A curve $C$ over $K$ is called {\it superspecial} if its Jacobian ${\rm Jac}(C)$
is isomorphic to a product of supersingular elliptic curves. 
This is equivalent to that the Cartier operator on $H^0(C,{\varOmega}_C)$ is zero (cf.~\cite[Theorem 4.1]{Nygaard}).
The matrix (called the Cartier--Manin matrix) representing the operator with respect to the basis $dx/y,xdx/y,\ldots,x^{g-1}dx/y$\ 
is described as
\[
\begin{pmatrix} 
  c_{p-1} & c_{p-2} & \dots  & c_{p-g} \\
  c_{2p-1} & c_{2p-2} & \dots  & c_{2p-g} \\
  \vdots & \vdots & \ddots & \vdots \\
  c_{gp-1} & c_{gp-2} & \dots  & c_{gp-g}
\end{pmatrix}
\]
where $c_k$ is the $x^k$-coefficient of $f(x)^{(p-1)/2}$, see \cite[p.~79]{Manin}.
We shall regard $\lambda_1,\ldots,\lambda_{2g-1}$ as indeterminates and consider  $c_{ip-j}$ as polynomials in $\lambda_1,\ldots,\lambda_{2g-1}$ for $1\le i,j\le g$.
The space $V(c_{ip-j}|1\le i,j\le g)$ consisting of points where $c_{ip-j}$ vanish for all $i,j$ with $1\leq i,j\le g$ is called the {\it superspecial locus}, which is known to consist of finite points.
The aim of this paper is to prove the ideal generated by $c_{ip-j}$ for $1\le i, j\le g$ is of multiplicity-one at any point of $V(c_{ip-j}|1\le i,j\le g)$.
To achieve this, we use the Lauricella hypergeometric series of type (D) in $2g-1$ variables and partial differential equations they satisfy.
The fact that entries of the Cartier--Manin matrix are described by truncations of Lauricella hypergeometric series is due to Ohashi and Harashita (\cite[Theorem 6.5]{OH}), which we review in Definition~\ref{def:Truncation} and Fact~\ref{fact:transition}.
Throughout this paper, we denote by $\mathbb{N}_0$ the set of nonnegative integers, and we use a symbol $\an{x} = x(x+1)(x+2)\cdots(x+n-1)$ for $n \in \mathbb{N}_0$.
(We consider $\an[0]{x}=1$ for $n=0.$)
Lauricella hypergeometric series of type (D) in $2g-1$ variables is defined as follows (cf.~Definition \ref{def:Appell}):
\begin{equation}\label{eq:L-hyp-geom-series}
\ALseries
\coloneqq 
 \sum_{(n_1,\ldots,n_{2g-1})\in \mathbb{N}_0^{2g-1}} A_{n_1,\ldots,n_{2g-1}}\zzz,
\end{equation}
where
\begin{equation}
    A_{n_1,\ldots,n_{2g-1}}\coloneqq\subALseriesdef
\end{equation}
with $a,b_1,\ldots,b_{2g-1},c \in \mathbb{C},-c \notin \mathbb{N}_0$.
In this paper, the hypergeometric series is considered as an element of the ring of formal power series in $z_1,\ldots,z_{2g-1}$.

The hypergeometric series $w\coloneqq\ALseries$ satisfies the partial differential equations $\diff{\ell} w=0$ for $\ell=1,\ldots,2g-1$ 
and $\diffsimp{\ell}{m} w=0$ for $1\le \ell<m\le 2g-1$, where
\begin{eqnarray*}
\diff{\ell}
&=&z_\ell(1-z_\ell)\pdd{z_\ell}+\sum_{\substack{1 \le k \le 2g-1,\\k \neq \ell}}
{ z_k(1-z_\ell)\genpdd{z_\ell}{z_k} }  \notag\\
&&+\left(c-(a+b_\ell+1)z_\ell\right)\pd{z_\ell}
-\sum_{\substack{1 \le k \le 2g-1,\\k \neq \ell}}{ b_\ell z_k \pd{z_k} }-ab_\ell
\end{eqnarray*}
and
\begin{equation*}
\diffsimp{\ell}{m}\coloneqq
(z_\ell-z_m)\dfrac{\partial^2}{\partial z_\ell\partial z_m}
-b_m\dfrac{\partial}{\partial z_\ell}
+b_\ell\dfrac{\partial}{\partial z_m}.
\end{equation*}
Let $j\in\{1,\ldots,g\}$.
Let $\diffchar{\ell}{j}$ and $\diffsimp{\ell}{m}$, for $\ell,m\in\{1,\ldots,2g-1\}$ with $\ell<m$, denote the partial differential operators ${\mathcal D}_\ell$ and $\diffsimp{\ell}{m}$, respectively, associated with $a = (2g+1)/2 - j$, $(b_1,b_2,\ldots,b_{2g-1})=(1/2,1/2,\ldots,1/2)$, and $c = g-j+1$.

From the result (\cite[Theorem 6.5]{OH}) of Ohashi--Harashita we mentioned above, we naturally expect that the operators $\diffchar{\ell}{j}$ and $\diffsimp{\ell}{m}$ annihilate $c_{ip-j}$.
The first theorem asserts that this indeed holds:
\begin{restatable}{theor}{thmA}\label{thm:PDEinPositiveChar}
For every $i,j$ with $1\le i,j \le g$, the entry $c_{ip-j}$ of the Cartier--Manin matrix
satisfies
\begin{equation*}
\diffchar{\ell}{j} c_{ip-j} = 0
\end{equation*}
for $\ell=1,2,\ldots,2g-1$ and
\begin{equation*}
\diffsimp{\ell}{m} c_{ip-j} = 0
\end{equation*}
for $1\le \ell < m \le 2g-1$.
\end{restatable}

In the genus-one case, it is well-known (cf.~\cite[p.255]{Deuring} and \cite[Theorem V.4.1]{Silverman}) that the elliptic curve of Legendre form $ y^2 = x(x-1)(x-t)$ over $K$ 
is supersingular if and only if
$H_p(t) =0$, where $\displaystyle H_p(t) \coloneqq \sum_{i=0}^{(p-1)/2}{\binom{(p-1)/2}{i}}^2 t^i$, which is equal, up to the sign $(-1)^{(p-1)/2}$, to the coefficient of $x^{p-1}$ in the expression $\left( x(x-1)(x-t) \right)^{(p-1)/2}$.
The separability of the polynomial $H_p(t)$ was proved by Igusa \cite{Igusa}.
The key part of his proof lies in the use of a differential operator which vanishes $H_p(z)$:
$\mathop{\mathcal D} = z(1-z) d^2/dz^2 + (c-(a+b+1)z)d/dz-ab$
for $a=1/2$, $b=1/2$ and $c=1$.

For $a,b,c \in \mathbb{C}$ with $-c \notin \mathbb{N}_0$,
over $\mathbb C$, the differential operator $\mathop{\mathcal D}$ annihilates the Gauss hypergeometric series $\displaystyle F(a,b,c\,;z) = \sum_{n\in\mathbb{N}_0} \dfrac{\an{a}\an{b}}{\an{c}\an{1}}z^n$.
For the choice $a=1/2,b=1/2,c=1$, $H_p(t)$ equals the truncation of $F(1/2,1/2,1\,;z)$ by degree $m$.

For special families of genus-$2$ curves, similar results have been obtained by Ibukiyama, Katsura, and Oort~\cite[Proposition 1.14]{IKO}.
The Cartier--Manin matrix for the curve of genus two in Rosenhain form $y^2 = f(x)\coloneqq x(x-1)(x-\lambda_1)(x-\lambda_2)(x-\lambda_3)$ is
\[
\begin{pmatrix}
    c_{p-1} & c_{p-2}\\
    c_{2p-1} & c_{2p-2}
\end{pmatrix},
\]
and the result on multiplicity-one for this entire family of genus-two curves has been obtained by Harashita and Yamamoto~\cite{HY}.
In \cite{HY}, equalities involving partial derivatives of $c_{ip-j}$, called contiguity relations, are as follows; they played a key role in proving multiplicity-one(\cite[Theorem 5.3]{HY}).
\begin{align*}
\left(\sum_{k=1}^3 (z_k^2-z_k)\partial_k\right) c_{ip-1} &=-\dfrac{1}{2}(z_1+z_2+z_3-2)c_{ip-1}-\dfrac{1}{2}c_{ip-2},\\
\left(\sum_{k=1}^3 (1-z_k)\partial_k\right) c_{ip-2} &= \dfrac{1}{2}(c_{ip-1}+c_{ip-2}).
\end{align*}

In this paper, we prove more fundamental equalities (Theorem~\ref{thm:Contiguity_Relations} and Corollary~\ref{cor:Relations}) in the framework of arbitrary genus which generate the contiguity relations.
With the help of these equations and Theorem~\ref{thm:PDEinPositiveChar}, we obtain our main theorems:
\begin{restatable}{theor}{thmB}\label{thm:SspImpliesNonsing}
For any solution $(\lambda_1,\ldots,\lambda_{2g-1})$ of the equations
\[
c_{ip-j}(z_1,\ldots,z_{2g-1}) = 0
\]
with $1\le i,j \le g$, the curve $y^2 = x(x-1)(x-\lambda_1)(x-\lambda_2)\cdots(x-\lambda_{2g-1})$ is nonsingular.
More precisely, $0,1,\lambda_1,\ldots,\lambda_{2g-1}$ are distinct.
\end{restatable}

\begin{restatable}{theor}{thmC}\label{thm:MultiplicityOne}
The scheme defined by the ideal generated by $\setsymbol{c_{ip-j}}{1\le i,j \le g}$ of ${\mathbb F}_p[z_1,\ldots,z_{2g-1}]$ is reduced.
(This is equivalent to saying that the scheme is nonsingular.
It is well-known that the scheme is zero-dimensional.)
\end{restatable}

\begin{Rem}\label{rem:restriction_on_primes}
 If $g > (p-1)/2$, there is no superspecial hyperelliptic curve of genus $g$ in characteristic $p$ by Ekedahl \cite[Theorem 1.1 in Section 2]{Ekedahl}, whence the above theorems hold trivially.
 Hence we will assume $g \le (p-1)/2$ in the proof (cf.\ Section 3).
\end{Rem}

Theorem~\ref{thm:MultiplicityOne} can be regarded as a generalization of the result by Igusa \cite{Igusa} in the case of genus-one and the result by Harashita--Yamamoto \cite{HY} in the case of genus-two.
We call this theorem {\it the multiplicity-one theorem for the superspeciality of hyperelliptic curves $y^2=x(x-1)(x-z_1)(x-z_2)\cdots(x-z_{2g-1})$.}

This paper is organized as follows.
In Section 2, We review some facts about the Lauricella hypergeometric series and its truncation, in connection with Cartier--Manin matrices.
In Section 3, we collect fundamental properties of the Lauricella series related to superspeciality and study partial differential equations involving entries of a Cartier--Manin matrix, which will be used in the proof of the multiplicity-one theorem.
In Section 4, We employ induction on the genus to prove Theorem~\ref{thm:SspImpliesNonsing}, and subsequently Theorem~\ref{thm:MultiplicityOne} by using Theorem~\ref{thm:SspImpliesNonsing}, thereby deducing the multiplicity-one theorem for every genus.

\subsection*{Acknowledgements}
This paper forms part of the author's thesis.
The author expresses sincere gratitude to his advisor, Professor Shushi Harashita, for his invaluable guidance, encouragement, and patience throughout the course of this work.
Without such support, this achievement would not have been possible.
This work was supported by YNU-SPRING.

\section{Preliminaries}
In this section, we review the Lauricella hypergeometric series and the Cartier--Manin matrix.

\subsection{Lauricella hypergeometric series}
For an element $\bm{n}=(n_1,\ldots,n_{2g-1})$ of $\mathbb{N}_0^{2g-1}$, let $\lvert \bm{n} \rvert = n_1+\dots+n_{2g-1}$ denote the total degree of the monomial $z_1^{n_1} \dotsm z_{2g-1}^{n_{2g-1}}$.
We use the standard basis vectors $\bm{e}_k$ with $1$ in the $k$-th row and all other entries zero.

The Lauricella hypergeometric series in $g$ variables is defined as follows:
\begin{Def}\label{def:Appell}
Consider the series $\ALseries$:
\begin{equation}
\ALseriesdef \zzz \in \mathbb{C}\fps, \label{eq:Appell3}
\end{equation}
where $a,b_1,\ldots,b_{2g-1},c \in \mathbb{C},-c \notin \mathbb{N}_0$.
We call it the Lauricella hypergeometric series in $2g-1$ variables
with respect to $(a,b_1,\ldots,b_{2g-1},c)$.
\end{Def}

The Lauricella hypergeometric series satisfies the following partial differential equations,
see $E_D(a,(b),c)$ in \cite[Table 3.2 in Section 3.5]{Matsumoto}.

\begin{Thm}\label{thm:PDEinCharzero}
The Lauricella hypergeometric series in $2g-1$ variables $w\coloneqq\ALseries$ satisfies
\begin{equation}\label{eq:main}
    \diff{\ell}w=0
\end{equation}
for $\ell = 1,\ldots,2g-1$ and
\begin{equation}\label{eq:main2}
    \diffsimp{\ell}{m}w=0
\end{equation}
for $1\le \ell < m\le 2g-1$, where the partial differential operators $\diff{1},\ldots,\diff{2g-1}$ are given by
\begin{align*}
\diff{1}
=&z_1(1-z_1)\pdd{z_1}+z_2(1-z_1)\genpdd{z_1}{z_2}+\cdots+z_{2g-1}(1-z_1)\genpdd{z_1}{z_{2g-1}}\\
&+\left(c-(a+b_1+1)z_1\right)\pd{z_1}-b_1 z_2\pd{z_2}-\cdots-b_1 z_{2g-1}\pd{z_{2g-1}}-a b_1,
\end{align*}
\begin{align*}
\diff{2}
=&z_1(1-z_2)\genpdd{z_1}{z_2}+z_2(1-z_2)\pdd{z_2}
+z_3(1-z_2)\genpdd{z_2}{z_3}\\
&+\cdots+z_{2g-1}(1-z_2)\genpdd{z_2}{z_{2g-1}}-b_2 z_1\pd{z_1}\\
&+\left(c-(a+b_2+1)z_2\right)\pd{z_2}-b_2 z_3\pd{z_3}-\cdots-b_2 z_{2g-1}\pd{z_{2g-1}}-a b_2,
\end{align*}
\begin{center}
$\vdots$
\end{center}
\begin{align*}
\diff{2g-1}
=&z_1(1-z_{2g-1})\genpdd{z_1}{z_{2g-1}}+\cdots+z_{2g-2}(1-z_{2g-1})\genpdd{z_{2g-2}}{z_{2g-1}}\\
&+z_{2g-1}(1-z_{2g-1})\pdd{z_{2g-1}}-b_{2g-1} z_1\pd{z_1}-\cdots-b_{2g-1} z_{2g-2}\pd{z_{2g-2}}\\
&+\left(c-(a+b_{2g-1}+1)z_{2g-1}\right)\pd{z_{2g-1}}-a b_{2g-1},
\end{align*}
and 
\begin{equation}
\diffsimp{\ell}{m}\coloneqq
(z_\ell-z_m)\dfrac{\partial^2}{\partial z_\ell\partial z_m}
+b_\ell\dfrac{\partial}{\partial z_m}
-b_m\dfrac{\partial}{\partial z_\ell}
\end{equation}
for $1\le \ell < m \le 2g-1$.
\end{Thm}

\begin{Rem}\label{rem:StandardRelation}
Let $K$ be any field.
The partial differential operators $\diff{\ell}$ for $\ell = 1,\ldots,2g-1$ can be considered if $a,b_1,\ldots,b_{2g-1},c\in K$.
Let $F$ be any element of $K\fps$, say
\[
F = \sum_{\nnn\in\mathbb{N}_0} B_{\nnn}\zzz.
\]
A direct computation shows that the $\zzz$-coefficient of
$\diff{1} F$ is
\[
(1+n_1)(c+\lvert\bm{n}\rvert)B_{\bm{n}+\bm{e}_1} -(b_1+n_1)(a+\lvert\bm{n}\rvert) 
B_{\bm{n}},
\]
where $\bm{n}=(\nnn)$.
If the characteristic of $K$ is $0$ and $\diff{1}F=0$,
then we have the recurrence relation
\begin{equation}\label{eq:StandardRelation}
B_{\bm{n}+\bm{e}_1} = \frac{(b_1+n_1)(a+\lvert\bm{n}\rvert)}{(1+n_1)(c+\lvert\bm{n}\rvert) }B_{\bm{n}}.
\end{equation}
We call this the {\it standard recurrence relation} for $F$ and $\diff{1}$.
Similar things also hold for $\diff{2}F,\ldots,\diff{2g-1}F$.
This shows that if the characteristic of $K$ is $0$, then the partial differential equations \eqref{eq:main}
determine $F$ from $B_{0,\ldots,0}$ provided $B_{0,\ldots,0}\ne 0$.
\end{Rem}

\subsection{Cartier--Manin matrices and hypergeometric series}\label{subsec:CM-HG}
Let $p$ be an odd prime and $K$ an algebraically closed field of characteristic $p$.
Let $g$ be a natural number.
A hyperelliptic curve of genus $g$ over $K$ is realized by $y^2 = f(x)$ with separable polynomial $f(x) \in K[x]$ of degree $2g+1$.
Let us review its Cartier--Manin matrix.

\begin{Def}\label{def:CMMatrix}
For the hyperelliptic curve $C$ defined by $y^2 = f(x)$ with
\[
f(x)=x(x-1)(x-z_1)\cdots(x-z_{2g-1}),
\]
let
\[
f(x)^{\frac{p-1}{2}}=\sum_{k=0}^{\frac{p-1}{2}(2g+1)} {c_k x^k}.
\]
be the expansion of $f(x)^{\frac{p-1}{2}}$.
The Cartier--Manin matrix of $C$ is defined as
\[
M = { \Bigl( c_{ip-j} \Bigr) }_{1\le i,j\le g}
\]
(cf.\ \cite[Section 2]{Yui}).
\end{Def}

We review truncations of the Lauricella hypergeometric series and describe $c_{ip-j}$ in relation to them, following the work of Ohashi and Harashita \cite{OH}.
\begin{Def}[Ohashi--Harashita~{\cite[Definition 6.4 and Example 6.10]{OH}}]\label{def:Truncation}
For $1 \le i,j \le g$, we set $a' \coloneqq (2g+1)/2 - j
= g-j+1/2,\, b' \coloneqq 1/2,\, c' \coloneqq a'+1-1/2 = g-j+1,\, d' \coloneqq (p-1)(2g+1)/2 - ip + j$.
For the Lauricella hypergeometric series
\[
\mathcal{F}(a',b',\ldots,b',c'\,;z_1,\ldots,z_{2g-1})
= \sum_{n_1,\ldots,n_{2g-1}\in\mathbb{N}_0} A^{(i,j)}_{n_1,\ldots,n_{2g-1}} z_1^{n_1} \cdots z_{2g-1}^{n_{2g-1}}
\]
with
\begin{equation}\label{eq:Aij}
A^{(i,j)}_{n_1,\ldots,n_{2g-1}}
\coloneqq\frac{\an[\displaystyle\sum_{k=1}^{2g-1} n_k]{a'} \displaystyle\prod_{k=1}^{2g-1} \an[n_k]{b'}}{\an[\displaystyle\sum_{k=1}^{2g-1} n_k]{c'} \displaystyle\prod_{k=1}^{2g-1} \an[n_k]{1}},
\end{equation}
we define its truncation
$\truncatedAL$
to be  the sum of $z^{n_1} \cdots z^{n_{2g-1}}$-terms 
for $(2g-1)$-tuples $(\nnn)$ satisfying
\begin{equation}\label{eq:truncation_condition}
\begin{cases}
{\displaystyle n_k \le \dfrac{p-1}{2} \quad (k=1,\ldots,2g-1)},\\
{\displaystyle d'-\dfrac{p-1}{2} \le n_1+\cdots+n_{2g-1} \le d'. }
\end{cases}
\end{equation}
We denote the set of exponents of the monomials in $\truncatedAL$ by $\supp$, and Lemma~\ref{lem:support} below implies 
\begin{equation}\label{eq:supp}
\supp = \setsymbolin{\bm{n}}{\mathbb{N}_0^{2g-1}}{\begin{array}{l}
    n_k \le \dfrac{p-1}{2} \quad (k=1,\ldots,2g-1),\\
    d'-\dfrac{p-1}{2} \le \lvert\bm{n}\rvert \le d'
\end{array}}.
\end{equation}
A priori, $\truncatedAL$ is an element of $\mathbb{Q}[z_1,\ldots,z_{2g-1}]$ but since the denominator of the coefficient of each term of $\truncatedAL$ is coprime to $p$ (see Lemma~\ref{lem:support}), we regard $\truncatedAL$
as an element of ${\mathbb{F}_p}[z_1,\ldots,z_{2g-1}]$.
We will write $\widetilde{A}_{\nnn}^{(i,j)} \in \mathbb{F}_p$ for the $z_1^{n_1} \dotsm z_{2g-1}^{n_{2g-1}}$-coefficient of $\truncatedAL$ when we view $\truncatedAL$ as an element of ${\mathbb{F}_p}[z_1,\ldots,z_{2g-1}]$.
\end{Def}

\section{Lauricella series and partial differential equations on superspeciality}
We collect fundamental properties of the Lauricella series related to superspeciality and study partial differential equations involving $c_{ip-j}$.
Building on this foundation, we establish the multiplicity-one theorem in the subsequent section.
For an element $\bm{n}=(n_1,\ldots,n_{2g-1})$ of $\mathbb{N}_0^{2g-1}$, we use the notation $\bm{z}^{\bm{n}}=z_1^{n_1} \dotsm z_{2g-1}^{n_{2g-1}}$, and $\lvert \bm{n} \rvert = n_1+\dots+n_{2g-1}$ for its total degree.
We use the standard basis vectors $\bm{e}_k$ with $1$ in the $k$-th row and all other entries zero.

\subsection{Properties of Lauricella series on superspeciality}
We collect recurrence relations of ${A}_{\nnn}^{(i,j)}$, which will be used later on.
We also review the relation between the Cartier--Manin matrix of a hyperelliptic curve and the associated Lauricella hypergeometric series.

The following lemma is a straightforward generalization of \cite[Lemma 2.9]{HY}.
\begin{Lem}\label{lem:recurrence_relations}
Let $k \in \{1,2,\ldots,2g-1\}$.
\begin{enumerate}
\item[({\romannumeral 1})]
(Standard recurrence relation) We have
\[
{A}_{\bm{n}+\bm{e}_k}^{(i,j)}
=\frac{(g-j+\frac{1}{2}+\lvert\bm{n}\rvert)(\frac{1}{2}+n_k)}{(g-j+1+\lvert\bm{n}\rvert)(1+n_k)}{A}_{\bm{n}}^{(i,j)}
\]
for $n_1,\ldots,n_{2g-1}\ge0$.

\item[({\romannumeral 2})]
For $n_1,\ldots,n_{2g-1}\ge 0$, and for $j<g$, we have
\begin{equation}\label{eq:ALinterchanging}
{A}_{\bm{n}}^{(i,j)} = \frac{(g-j)(g-j-\frac{1}{2}+\lvert\bm{n}\rvert)}{(g-j-\frac{1}{2})(g-j+\lvert\bm{n}\rvert)} {A}_{\bm{n}}^{(i,j+1)}
\end{equation}
and
\begin{equation}\label{eq:ALinterchanging2}
{A}_{\bm{n}+\bm{e}_k}^{(i,j+1)}=\frac{(g-j-\frac{1}{2})(\frac{1}{2}+n_k)}{(g-j)(1+n_k)}{A}_{\bm{n}}^{(i,j)}. 
\end{equation}
\end{enumerate}
\end{Lem}

The following lemma concerns the support of the truncation $\truncatedAL$.
We recall that we assume $g \le (p-1)/2$ (see Remark~\ref{rem:restriction_on_primes}).
\begin{Lem}\label{lem:support}
For any $(2g-1)$-tuple $(\nnn)$ which satisfies the condition \eqref{eq:truncation_condition}, $\widetilde{A}_{\bm{n}}^{(i,j)}$ is nonzero.
\end{Lem}

\begin{proof}
We note that it suffices to check the claim for any $(2g-1)$-tuple $\bm{n}$ such that $\lvert\bm{n}\rvert$ is minimal among all $(\nnn)\in\mathbb{N}_0^{2g-1}$ satisfying \eqref{eq:truncation_condition}.
In fact, we use the standard recurrence relation (Lemma~\ref{lem:recurrence_relations}({\romannumeral 1})) inductively on $\lvert\bm{n}\rvert$, and the fact that $(g-j+1/2+\lvert\bm{n}\rvert)/(g-j+1+\lvert\bm{n}\rvert)$ is nonzero for any $\bm{n}=(\nnn)\in\mathbb{N}_0^{2g-1}$ with $d'-\dfrac{p-1}{2} \le \lvert\bm{n}\rvert \le d'-1$.
To verify this fact, note that for such $\bm{n}$
\[
(g-i)p+\frac{1}{2} \le \lvert\bm{n}\rvert+a' \le (g-i)p+\frac{p}{2}-1,
\]
which implies that $\lvert\bm{n}\rvert+a'=g-j+\frac{1}{2}+\lvert\bm{n}\rvert$ is, over $\mathbb{F}_p$, one of those in $\{1/2,3/2,\ldots,(p-2)/2\}$ and $g-j+1+\lvert\bm{n}\rvert$ is one of those in $\{1,2,\ldots,(p-1)/2\}$.
Hence $(g-j+1/2+\lvert\bm{n}\rvert)/(g-j+1+\lvert\bm{n}\rvert))$ is nonzero.

Thus, we take any $\bm{n}$ with $\lvert\bm{n}\rvert$ minimal.
When $\bm{n}=0$, we have $\widetilde{A}_{0}^{(i,j)}=1$.
Suppose that $\bm{n}\ne0$.
In this case, by minimality we have $\lvert\bm{n}\rvert = d'-(p-1)/2$, and hence $\lvert\bm{n}\rvert+a' = (g-i)p+1/2$.
By \eqref{eq:Aij}, it suffices to show that $\an[\lvert\bm{n}\rvert]{a'} / \an[\lvert\bm{n}\rvert]{c'}$ is nonzero.
\begin{align*}
\frac{\an[\lvert\bm{n}\rvert]{a'}}{\an[\lvert\bm{n}\rvert]{c'}} &= \frac{(g-j+\frac{1}{2})(g-j+\frac{3}{2})\cdots((g-i)p-\frac{1}{2})}{(g-j+1)(g-j+2)\cdots(g-i)p}\\
&= \frac{(2(g-j)+1)(2(g-j)+3)\cdots(2(g-i)p-1)}{2^{\lvert\bm{n}\rvert} (g-j+1)(g-j+2)\cdots(g-i)p}\\
&= \frac{(2(g-i)p-1)!!}{((g-i)p)!} \cdot \frac{(g-j)!}{2^{\lvert\bm{n}\rvert} (2(g-j)-1)!!}.
\end{align*}
Since $g-j \le g-1 \le m-1 = (p-3)/2$, we deduce that $(g-j)!$ and $(2(g-j)-1)!!$ are nonzero in $\mathbb{F}_p$.
It remains to check that $(2(g-i)p-1)!!/((g-i)p)!$ is nonzero.
The subset of $p$-divisible numbers in $\{2(g-i)p-1,2(g-i)p-3,\ldots,1\}$ is $\{ (2(g-i)-1)p,(2(g-i)-3)p,\ldots,p \}$, whose size is  $g-i$.
The subset of $p$-divisible numbers in $\{(g-i)p,(g-i)p-1,\ldots,1\}$ is $\{ (g-i)p,(g-i-1)p,\ldots,p \}$, whose size is  $g-i$.
Therefore, in the expression of $(2(g-i)p-1)!!/((g-i)p)!$, the p-divisible numbers are canceled, and this value is nonzero in $\mathbb{F}_p$.
\end{proof}

The following relation between the entries of the Cartier--Manin matrices and the truncations $\truncatedAL$ of the Lauricella hypergeometric series is known.
This is a special case of the result \cite[Theorem 6.5]{OH} by Ohashi and Harashita.
\begin{Fact}\label{fact:transition}
Assume $p>2$.
Over $\mathbb{F}_p$, we have
\[
c_{ip-j} = \frac{\an[d']{c'}}{\an[d']{a'}} \truncatedAL
\]
for any $(i,j)$ with $1 \le i,j \le g$.
\end{Fact}

The following lemma provides a concrete description of the constant in Fact~\ref{fact:transition}.
\begin{Lem}\label{lem:simplification}
If $g\le (p-1)/2+1$ (cf. Remark~\ref{rem:restriction_on_primes}), then over $\mathbb{F}_p$, the constant $\an[d']{c'} / \an[d']{a'}$ is expressed as
\[
\frac{ \an[d']{c'} }{ \an[d']{a'} } = (-1)^\frac{p-1}{2}\frac{(2g-2i)!!}{(2g-2i-1)!!}\frac{(2g-2j-1)!!}{(2g-2j)!!},
\]
and this is nonzero.
\end{Lem}
\begin{proof}
By \cite[Example 6.9]{OH}, we know that
\begin{equation}
\frac{ \an[d']{c'} }{ \an[d']{a'} }  = \frac{(p(2g-2i+1)-1)!!}{(p(2g-2i+1)-2)!!}\frac{(2g-2j-1)!!}{(2g-2j)!!}.
\end{equation}
Hence, it suffices to show
\begin{equation}\label{eq:simplification_to_be_shown}
\frac{(p(2g-2i+1)-1)!!}{(p(2g-2i+1)-2)!!} = (-1)^\frac{p-1}{2}\frac{(2g-2i)!!}{(2g-2i-1)!!}.
\end{equation}

Since $p(2g-2i+1)-1$ is even, $(p(2g-2i+1)-1)!!$ is the product of the elements of $A\coloneqq\{p(2g-2i+1)-1,p(2g-2i+1)-3,\ldots,2\}$, whose size is $(p(2g-2i+1)-1)/2 = p(g-i) + (p-1)/2$.
The subset of $p$-divisible numbers in this set is $A'\coloneqq\{p(2g-2i),p(2g-2i-2),\ldots,2p\}$, whose size is $g-i$.
Similarly, $(p(2g-2i+1)-2)!!$ is the product of the elements of $B\coloneqq\{p(2g-2i+1)-2,p(2g-2i+1)-4,\ldots,1\}$, whose size is $p(g-i) + (p-1)/2$, and the subset of the $p$-divisible numbers contained in $B$ is $B'\coloneqq\{p(2g-2i-1),p(2g-2i-3),\ldots,p\}$ and its size is $g-i$.

In addition, observe that over $\mathbb{F}_p$, multiplication by $-1$ maps $A$ to $B$, and hence $A \setminus A'$ to $B \setminus B'$.
Moreover, $\# (A \setminus A') = \# (B \setminus B') = p(g-i) + (p-1)/2 - (g-i) = (p-1)(g-i) + (p-1)/2 \equiv (p-1)/2 \pmod{2}$.
From these considerations, we obtain \eqref{eq:simplification_to_be_shown}:
\[
\frac{(p(2g-2i+1)-1)!!}{(p(2g-2i+1)-2)!!}
= (-1)^\frac{p-1}{2} \frac{p^{g-i} (2g-2i)!!}{p^{g-i} (2g-2i-1)!!}
= (-1)^\frac{p-1}{2}\frac{(2g-2i)!!}{(2g-2i-1)!!}.
\]

When $g\le (p-1)/2+1$, we have $2g-2i \le 2g-2 \le p-1$, and similarly $2g-2j \le p-1$.
These inequalities establish the claim of well-definedness and non-vanishing.
\end{proof}

\subsection{Relations among partial derivatives of $c_{ip-j}$}
Let $j\in\{1,\ldots,g\}$.
Let $\diffchar{\ell}{j}$ and $\diffsimp{\ell}{m}$ for $\ell,m\in\{1,\ldots,2g-1\}$ with $\ell< m$ be the  partial differential operators ${\mathcal D}_\ell$ and $\diffsimp{\ell}{m}$, respectively, in Theorem~\ref{thm:PDEinCharzero} associated with $a = (2g+1)/2 - j$, $(b_1,b_2,\ldots,b_{2g-1})=(1/2,1/2,\ldots,1/2)$, and $c = g-j+1$.
We have the following partial differential equations satisfied by $c_{ip-j}$.
They constitute a direct generalization of \cite[Theorem A]{HY}. 

\thmA*

\begin{proof}
To show the first equality, it suffices to show that $\diffchar{\ell}{j} \truncatedAL = 0$ by Fact~\ref{fact:transition}.
From the symmetry for $\ell,$ we prove only the case of $\ell=1$.

Let $\widetilde{A}_{\nnn}^{(i,j)}$ be the $\zzz$-coefficients of $\truncatedAL$. Then
\[
\widetilde{A}_{\bm{n}}^{(i,j)} = A_{\bm{n}}^{(i,j)}
=\frac{\an[\lvert\bm{n}\rvert]{a'} \an[n_1]{b'}\cdots\an[n_{2g-1}]{b'}}
{\an[\lvert\bm{n}\rvert]{c'} \an[n_1]{1}\cdots\an[n_{2g-1}]{1}}
\]
for $\bm{n}=(\nnn)\in\supp,$ and $\widetilde{A}_{\bm{n}}^{(i,j)}=0$ otherwise.
By Remark \ref{rem:StandardRelation}, we have
\[
\diffchar{1}{j} \truncatedAL
=\sum_{(n_1,\ldots,n_{2g-1})\in \mathbb{N}_0^{2g-1}} C_{\nnn}^{(i,j)} \zzz,
\]
where $B_{\bm{n}}^{(i,j)} \coloneqq (1+n_1)(c'+\lvert\bm{n}\rvert) \widetilde{A}_{\bm{n}+\bm{e}_1}^{(i,j)} -(b'+n_1)(a'+\lvert\bm{n}\rvert) \widetilde{A}_{\bm{n}}^{(i,j)}$ for $\bm{n}=(\nnn)\in\mathbb{N}_0^{2g-1}.$
We check that $B_{\bm{n}}^{(i,j)}$ vanishes for every $\bm{n}\in\mathbb{N}_0^{2g-1}$, dividing into five cases.

\underline{Case 1.} We consider the case where $\nnn$ satisfies
\[
n_1\le\frac{p-1}{2}-1,\,n_k\le\frac{p-1}{2}\,(2\le k \le 2g-1),\,d'-\frac{p-1}{2}\le \lvert\bm{n}\rvert \le d'-1.
\]
Since $n_1+1\le (p-1)/2$ and $d'-(p-1)/2+1\le (n_1+1)+n_2+\cdots+n_{2g-1}\le d',$
we have $\bm{n},\bm{n}+\bm{e}_1\in\supp,$ and by the standard recurrence relation, $B_{\bm{n}}^{(i,j)}=0.$

\underline{Case 2.} We consider the case where $n_1 = (p-1)/{2}$.
Since $n_1+1>(p-1)/2$ deduces $\widetilde{A}_{\bm{n}+\bm{e}_1}^{(i,j)}=0,$ and $b'+n_1=p/2=0,$ it holds $B_{\bm{n}}^{(i,j)}=0.$

\underline{Case 3.}  We consider the case where $\lvert\bm{n}\rvert = d'-(p-1)/{2}-1$. First, note that $\lvert\bm{n}\rvert < d'-(p-1)/2$ deduces $\widetilde{A}_{\bm{n}}^{(i,j)}=0.$ 
Note $a' + d' = 0$ in ${\mathbb F}_p$.
Since $c'+\lvert\bm{n}\rvert = (a'+1/2)+ ( d'-(p-1)/2-1 ) = a'+d' = 0,$ we have $B_{\bm{n}}^{(i,j)}=0.$

\underline{Case 4.}  We consider the case where $\lvert\bm{n}\rvert = d'$. Since $d' < (1+n_1)+n_2+\cdots+n_{2g-1}$ deduces  $\widetilde{A}_{\bm{n}+\bm{e}_1}^{(i,j)}=0$ and $a'+\lvert\bm{n}\rvert = a'+d' = 0,$ we have $B_{\bm{n}}^{(i,j)}=0.$

\underline{Case 5.} We consider the case where $\nnn$ satisfies $(p-1)/2 < n_k$ for some $k \in \{1,\ldots,2g-1\}$ or $ \lvert\bm{n}\rvert\le d'-(p-1)/2-2$ or $ d'+1 \le \lvert\bm{n}\rvert.$
In this case, neither $\bm{n}$ nor $\bm{n}+\bm{e}_1$ belongs to $\supp.$
Therefore, $\widetilde{A}_{\bm{n}}^{(i,j)}=\widetilde{A}_{\bm{n}+\bm{e}_1}^{(i,j)}=0,$ and we have $B_{\bm{n}}^{(i,j)}=0.$

To show the second equality, we only need to show $\diffsimp{\ell}{m}\truncatedAL=0$, since $c_{ip-j} = (-1)^\frac{p-1}{2}\cdot\dfrac{j}{i} \truncatedAL$.
From the symmetry for $(\ell, m)$, it suffices to consider the case $(\ell,m)=(1,2)$. 
The $\zzz$-coefficient of
\begin{align*}
&\diffsimp{1}{2}\truncatedAL\\
=&\sum_{(n_1,\ldots,n_{2g-1})\in \mathbb{N}_0^{2g-1}} \Bigl( n_1 n_2(z^{n_1}_1 z^{n_2-1}_2 z^{n_3}_3\cdots z^{n_{2g-1}}_{2g-1}-z^{n_1-1}_1 z^{n_2}_2\cdots z^{n_{2g-1}}_{2g-1})\\
&\hspace{4em} +\frac{1}{2}n_2z^{n_1}_1 z^{n_2-1}_2 z^{n_3}_3\cdots z^{n_{2g-1}}_{2g-1}-\frac{1}{2}n_1z^{n_1-1}_1 z^{n_2}_2\cdots z^{n_{2g-1}}_{2g-1}\Bigr) \widetilde{A}_{\nnn}^{(i,j)}\\
=&\sum_{\bm{n}\in\supp} C_{\bm{n}}^{(i,j)}\bm{z}^{\bm{n}},
\end{align*}
where $C_{\bm{n}}^{(i,j)} \coloneqq \left(n_1+1/2\right) \left(n_2+1\right)\widetilde{A}_{\bm{n}+\bm{e}_2}^{(i,j)}-\left(n_1+1\right)\left(n_2+1/2\right)\widetilde{A}_{\bm{n}+\bm{e}_1}^{(i,j)}$ for $\bm{n}=(\nnn)\in\mathbb{N}_0^{2g-1}.$
We check that $C_{\bm{n}}^{(i,j)}$ vanishes for every $\bm{n}\in\mathbb{N}_0^{2g-1}$, dividing into three cases.

\underline{Case 1.} We consider the case where $\nnn$ satisfies
\[
n_1\le\frac{p-1}{2}-1,\,n_2\le\frac{p-1}{2}-1,\,n_k\le\frac{p-1}{2}\,(3\le k \le 2g-1),\,d'-\frac{p-1}{2}-1\le \lvert\bm{n}\rvert \le d'-1.
\]
In this case, we have $\bm{n}+\bm{e}_1,\bm{n}+\bm{e}_2\in\supp,$ and by the standard recurrence relations in Lemma \ref{lem:recurrence_relations}({\romannumeral 1}) for $n_2$ and $n_1$, $C_{\bm{n}}^{(i,j)}=0.$

\underline{Case 2.} We consider the case where $n_1 = (p-1)/{2}$ or $n_2 = (p-1)/{2}$.
When $n_1 = (p-1)/{2}$, since $n_1+1/2=0$ and $\widetilde{A}_{\bm{n}+\bm{e}_1}^{(i,j)}=0$ by $\bm{n}+\bm{e}_1 \notin \supp$, $C_{\bm{n}}^{(i,j)}=0.$
When $n_2 = (p-1)/{2}$, we employ the same argument.

\underline{Case 3.} We consider the case where $\nnn$ satisfies $(p-1)/2 < n_k$ for some $k \in \{1,\ldots,2g-1\}$ or $ \lvert\bm{n}\rvert\le d'-(p-1)/2-2$ or $ d' \le \lvert\bm{n}\rvert.$
In this case, neither $\bm{n}+\bm{e}_1$ nor $\bm{n}+\bm{e}_2$ belongs to $\supp.$
Thus, $\widetilde{A}_{\bm{n}+\bm{e}_1}^{(i,j)}=\widetilde{A}_{\bm{n}+\bm{e}_2}^{(i,j)}=0,$ and we have $C_{\bm{n}}^{(i,j)}=0.$
\end{proof}

Next, we examine fundamental relations involving $c_{ip-j}$ under partial differential operators.
These yield the equalities (called contiguity relations) in \cite[Theorem 5.3]{HY} in the genus-two case.
\begin{Thm}\label{thm:Contiguity_Relations}
Set $m\coloneqq(p-1)/2$.
\begin{enumerate}
\item[({\romannumeral 1})]
For $k \in \{1,\ldots,2g-1\}, j<g,$ and for any $i$, we have
\[
z_k\partial_k c_{ip-j} -m c_{ip-j} = \partial_k\ c_{ip-(j+1)}.
\]

\item[({\romannumeral 2})]
For $j<g$ and any $i$, we have
\[
\left(\sum_{k=1}^{2g-1} (z_k-1)\partial_k\right) c_{ip-(j+1)} = (m-j+1) c_{ip-j} + (m-g+j+1) c_{ip-(j+1)}.
\]
\end{enumerate}
\end{Thm}
\begin{proof}
For a subset $S$ of $\mathbb{N}_0^{2g-1}$ we denote the set $\{\bm{n}+\bm{s} | \bm{s}\in S\}$ by $\bm{n}+S$,
which will be used for $S=\supp$, see \eqref{eq:supp} for the definition of $\supp$.

Note that
\[
\widetilde{A}_{\bm{n}}^{(i,j)} = A_{\bm{n}}^{(i,j)}
=\frac{\an[\lvert\bm{n}\rvert]{a'} \an[n_1]{b'}\cdots\an[n_{2g-1}]{b'}}
{\an[\lvert\bm{n}\rvert]{c'} \an[n_1]{1}\cdots\an[n_{2g-1}]{1}}
\]
for $\bm{n}=(n_1,\ldots,n_{2g-1})\in\supp,$ while $\widetilde{A}_{\bm{n}}^{(i,j)}=0$ otherwise.

({\romannumeral 1}) Let $\varphi \in {\mathbb{F}_p}[z_1,\ldots,z_{2g-1}]$ be
\[
\frac{\an[d']{a'}}{\an[d']{c'}} \left(z_k\partial_k c_{ip-j} -m c_{ip-j} - \partial_k\ c_{ip-(j+1)}\right).
\]
Since $\an[d']{a'}/ \an[d']{c'}$ is nonzero according to Lemma~\ref{lem:simplification}, it suffices to show that $\varphi$ is zero.
We transition to the Lauricella series via the link in Fact \ref{fact:transition}. $\varphi$ is
\begin{align}
& \sum_{\bm{n}\in\supp} n_k\widetilde{A}_{\bm{n}}^{(i,j)}\bm{z}^{\bm{n}}
-m \sum_{\bm{n}\in\supp} \widetilde{A}_{\bm{n}}^{(i,j)}\bm{z}^{\bm{n}} - \frac{c'-1}{a'-1} \sum_{\bm{n}\in\supp[i][j+1]} n_k\widetilde{A}_{\bm{n}}^{(i,j+1)}\bm{z}^{\bm{n}-\bm{e}_k}\notag\\
=& \sum_{\bm{n}\in\supp} \left(\frac{1}{2}+n_k\right)\widetilde{A}_{\bm{n}}^{(i,j)}\bm{z}^{\bm{n}} - \frac{g-j}{g-j-\frac{1}{2}} \sum_{\bm{n}\in\supp[i][j+1]} n_k\widetilde{A}_{\bm{n}}^{(i,j+1)}\bm{z}^{\bm{n}-\bm{e}_k}\notag\\
=& \sum_{\bm{n}\in\supp} \left(\frac{1}{2}+n_k\right)\widetilde{A}_{\bm{n}}^{(i,j)}\bm{z}^{\bm{n}} - \frac{g-j}{g-j-\frac{1}{2}} \sum_{\bm{n}\in\mathsf{T}^{(i,j+1)}_k} n_k\widetilde{A}_{\bm{n}}^{(i,j+1)}\bm{z}^{\bm{n}-\bm{e}_k},\label{eq:simple_contiguity_to_be_shown}
\end{align}
where $\mathsf{T}^{(i,j+1)}_k \coloneqq \supp[i][j+1] \setminus \setsymbolin{\bm{n}}{\supp[i][j+1]}{n_k=0}.$
By Lemma~\ref{lem:recurrence_relations}~({\romannumeral 2})\eqref{eq:ALinterchanging2}, we write \eqref{eq:simple_contiguity_to_be_shown} as
\begin{align}
&\sum_{\bm{n}\in\supp} \left(\frac{1}{2}+n_k\right)\widetilde{A}_{\bm{n}}^{(i,j)}\bm{z}^{\bm{n}} - \frac{g-j}{g-j-\frac{1}{2}} \sum_{\bm{n}\in-\bm{e}_k+\mathsf{T}^{(i,j+1)}_k} (n_k+1)\widetilde{A}_{\bm{n}+\bm{e}_k}^{(i,j+1)}\bm{z}^{\bm{n}}\notag\\
=& \sum_{\bm{n}\in\supp} \left(\frac{1}{2}+n_k\right)\widetilde{A}_{\bm{n}}^{(i,j)}\bm{z}^{\bm{n}} -  \sum_{\bm{n}\in-\bm{e}_k+\mathsf{T}^{(i,j+1)}_k} \left(\frac{1}{2}+n_k\right)\widetilde{A}_{\bm{n}}^{(i,j)}\bm{z}^{\bm{n}}.\label{eq:simple_contiguity_to_be_shown_2}
\end{align}
Now we note that, by setting $d'\coloneqq(p-1)(2g+1)/2-ip+j$,
\begin{align*}
\supp &= \setsymbolin{\bm{n}}{\mathbb{N}_0^{2g-1}}{\begin{array}{l}
    n_\ell \le \dfrac{p-1}{2} \quad (\ell=1,\ldots,2g-1),\\
    d'-\dfrac{p-1}{2} \le \lvert\bm{n}\rvert \le d'
\end{array}},\\
\supp[i][j+1] &= \setsymbolin{\bm{n}}{\mathbb{N}_0^{2g-1}}{\begin{array}{l}
    n_\ell \le \dfrac{p-1}{2} \quad (\ell=1,\ldots,2g-1),\\
    d'-\dfrac{p-1}{2}+1 \le \lvert\bm{n}\rvert \le d'+1
\end{array}}.
\end{align*}
Thus, $\left(-\bm{e}_k+\mathsf{T}^{(i,j+1)}_k\right) \setminus \supp = \emptyset$ and
\[
\supp \setminus \left(-\bm{e}_k+\mathsf{T}^{(i,j+1)}_k\right) = \setsymbolin{\bm{n}}{\supp}{n_k = (p-1)/2}.
\]
Therefore, \eqref{eq:simple_contiguity_to_be_shown_2} vanishes.

({\romannumeral 2}) Let $\psi \in {\mathbb{F}_p}[z_1,\ldots,z_{2g-1}]$ be
\[
\frac{\an[d']{a'}}{\an[d']{c'}} \left(\left(\sum_{k=1}^{2g-1} (z_k-1)\partial_k\right) c_{ip-(j+1)} - (m-j+1) c_{ip-j} - (-mg+j+1) c_{ip-(j+1)}\right),
\]
and we show that $\psi$ is zero.
Transitioning to the Lauricella series by Fact \ref{fact:transition}, $\psi$ is
\begin{align}
& \frac{c'-1}{a'-1} \sum_{k=1}^{2g-1} \sum_{\bm{n}\in\supp[i][j+1]} \left( n_k\widetilde{A}_{\bm{n}}^{(i,j+1)}\bm{z}^{\bm{n}} -n_k\widetilde{A}_{\bm{n}}^{(i,j+1)}\bm{z}^{\bm{n}-\bm{e}_k} \right)\notag\\
&+\left( j-\frac{1}{2} \right) \sum_{\bm{n}\in\supp} \widetilde{A}_{\bm{n}}^{(i,j)}\bm{z}^{\bm{n}} +\frac{c'-1}{a'-1} \left( g-j-\frac{1}{2} \right) \sum_{\bm{n}\in\supp[i][j+1]} \widetilde{A}_{\bm{n}}^{(i,j+1)}\bm{z}^{\bm{n}}\notag\\
=& \frac{c'-1}{a'-1} \sum_{k=1}^{2g-1} \sum_{\bm{n}\in\supp[i][j+1]}  n_k\widetilde{A}_{\bm{n}}^{(i,j+1)}\bm{z}^{\bm{n}} -\sum_{k=1}^{2g-1} \sum_{\bm{n}\in-\bm{e}_k+\mathsf{T}^{(i,j+1)}_k} \frac{c'-1}{a'-1} (n_k+1)\widetilde{A}_{\bm{n}+\bm{e}_k}^{(i,j+1)}\bm{z}^{\bm{n}}\notag\\
&+\left( j-\frac{1}{2} \right) \sum_{\bm{n}\in\supp} \widetilde{A}_{\bm{n}}^{(i,j)}\bm{z}^{\bm{n}} +\frac{c'-1}{a'-1} \left( g-j-\frac{1}{2} \right) \sum_{\bm{n}\in\supp[i][j+1]} \widetilde{A}_{\bm{n}}^{(i,j+1)}\bm{z}^{\bm{n}}\label{eq:second_contiguity_to_be_shown}
\end{align}
We calculate some parts of \eqref{eq:second_contiguity_to_be_shown} separately.
Using Lemma~\ref{lem:recurrence_relations}~({\romannumeral 2})\eqref{eq:ALinterchanging2}, we have
\begin{align}
& -\sum_{k=1}^{2g-1} \sum_{\bm{n}\in-\bm{e}_k+\mathsf{T}^{(i,j+1)}_k} \frac{c'-1}{a'-1} (n_k+1)\widetilde{A}_{\bm{n}+\bm{e}_k}^{(i,j+1)}\bm{z}^{\bm{n}} +\left( j-\frac{1}{2} \right) \sum_{\bm{n}\in\supp} \widetilde{A}_{\bm{n}}^{(i,j)}\bm{z}^{\bm{n}}\notag\\
=& -\sum_{k=1}^{2g-1} \sum_{\bm{n}\in-\bm{e}_k+\mathsf{T}^{(i,j+1)}_k} \left(\frac{1}{2}+n_k\right)\widetilde{A}_{\bm{n}}^{(i,j)}\bm{z}^{\bm{n}} +\left( j-\frac{1}{2} \right) \sum_{\bm{n}\in\supp} \widetilde{A}_{\bm{n}}^{(i,j)}\bm{z}^{\bm{n}}.\label{eq:second_contiguity_computation_1}
\end{align}
Here, we replace $-\bm{e}_k+\mathsf{T}^{(i,j+1)}_k$ with $\supp$ for the same reason mentioned in the last part of  ({\romannumeral 1}).
\eqref{eq:second_contiguity_computation_1} is
\begin{align}
& -\sum_{k=1}^{2g-1} \sum_{\bm{n}\in\supp} \left(\frac{1}{2}+n_k\right)\widetilde{A}_{\bm{n}}^{(i,j)}\bm{z}^{\bm{n}} +\left( j-\frac{1}{2} \right) \sum_{\bm{n}\in\supp} \widetilde{A}_{\bm{n}}^{(i,j)}\bm{z}^{\bm{n}}\notag\\
=&  -\sum_{\bm{n}\in\supp} \left(\frac{1}{2}(2g-1)+\lvert\bm{n}\rvert\right)\widetilde{A}_{\bm{n}}^{(i,j)}\bm{z}^{\bm{n}} +\left( j-\frac{1}{2} \right) \sum_{\bm{n}\in\supp} \widetilde{A}_{\bm{n}}^{(i,j)}\bm{z}^{\bm{n}}\notag\\
=&  -\sum_{\bm{n}\in\supp} \left(g-j+\lvert\bm{n}\rvert\right)\widetilde{A}_{\bm{n}}^{(i,j)}\bm{z}^{\bm{n}}.\label{eq:second_contiguity_computation_2}
\end{align}
Next, for the other part of \eqref{eq:second_contiguity_to_be_shown} we have
\begin{align}
& \frac{c'-1}{a'-1} \sum_{k=1}^{2g-1} \sum_{\bm{n}\in\supp[i][j+1]}  n_k\widetilde{A}_{\bm{n}}^{(i,j+1)}\bm{z}^{\bm{n}} +\frac{c'-1}{a'-1} \left( g-j-\frac{1}{2} \right) \sum_{\bm{n}\in\supp[i][j+1]} \widetilde{A}_{\bm{n}}^{(i,j+1)}\bm{z}^{\bm{n}}\notag\\
=& \frac{c'-1}{a'-1} \sum_{\bm{n}\in\supp[i][j+1]}  \lvert\bm{n}\rvert\widetilde{A}_{\bm{n}}^{(i,j+1)}\bm{z}^{\bm{n}} +\frac{c'-1}{a'-1} \left( g-j-\frac{1}{2} \right) \sum_{\bm{n}\in\supp[i][j+1]} \widetilde{A}_{\bm{n}}^{(i,j+1)}\bm{z}^{\bm{n}}\notag\\
=& \frac{c'-1}{a'-1} \sum_{\bm{n}\in\supp[i][j+1]} \left( g-j-\frac{1}{2}+\lvert\bm{n}\rvert \right)\widetilde{A}_{\bm{n}}^{(i,j+1)}\bm{z}^{\bm{n}}\notag\\
=& \sum_{\bm{n}\in\mathsf{U}^{(i,j+1)}} \left(g-j+\lvert\bm{n}\rvert\right) \widetilde{A}_{\bm{n}}^{(i,j)}\bm{z}^{\bm{n}}.\label{eq:second_contiguity_computation_3}
\end{align}
In the last line above, we used Lemma~\ref{lem:recurrence_relations}({\romannumeral 2})\eqref{eq:ALinterchanging} and defined
\begin{align*}
\mathsf{U}^{(i,j+1)} &\coloneqq \supp \cap \supp[i][j+1] \\
 &= \supp[i][j+1] \setminus \left( \supp[i][j+1]\setminus\supp \right)\\
&= \supp[i][j+1] \setminus \setsymbolin{\bm{n}}{\mathbb{N}_0^{2g-1}}{\begin{array}{l}
    n_\ell \le \dfrac{p-1}{2} \quad (\ell=1,\ldots,2g-1),\\
    \lvert\bm{n}\rvert = d'+1 = -g+j+\dfrac{1}{2}
\end{array}},
\end{align*}
which corresponds to almost all $\supp[i][j+1]$ excluding those $\bm{n}$ with maximal degree, for which $g-j-1/2+\lvert\bm{n}\rvert$ is zero.
Note
\[
\supp\setminus\supp[i][j+1]
= \setsymbolin{\bm{n}}{\mathbb{N}_0^{2g-1}}{\begin{array}{l}
    n_\ell \le \dfrac{p-1}{2} \quad (\ell=1,\ldots,2g-1),\\
    \lvert\bm{n}\rvert = d'-\dfrac{p-1}{2} = (g-i)p-g+j
\end{array}},
\]
which implies that the portion of $\left( \supp\setminus\supp[i][j+1] \right)$ in \eqref{eq:second_contiguity_computation_2} vanishes, and hence the sum of \eqref{eq:second_contiguity_computation_2} and \eqref{eq:second_contiguity_computation_3} also vanishes.
\end{proof}

\begin{Cor}\label{cor:Relations}
Set $m\coloneqq(p-1)/2.$
For $j<g$ and any $i$, we have
\[
\left(\sum_{k=1}^{2g-1} z_k (z_k-1)\partial_k\right) c_{ip-j}
= \left( m\sum_{k=1}^{2g-1}z_k +g -j \right) c_{ip-j} + (m-g+j+1) c_{ip-(j+1)},
\]
and for $r,s$ with $2 \le r \le s$ we have
\begin{align*}
&\left(\sum_{k=1}^{2g-1} z_k^r (z_k-1)\partial_k\right) c_{ip-(s+1-r)}\\
=& m\sum_{\ell=1}^{r-1}\sum_{k=1}^{2g-1} z_k^{r-\ell} (z_k-1)c_{ip-(s+\ell-r)} + \left( m\sum_{k=1}^{2g-1}z_k +g -s \right) c_{ip-s}\\
&+(m-g+s+1) c_{ip-(s+1)}.
\end{align*}
\end{Cor}
\begin{proof}
From the equalities ({\romannumeral 1}) and ({\romannumeral 2}) in Theorem~\ref{thm:Contiguity_Relations}, we have
\begin{align*}
& \left(\sum_{k=1}^{2g-1} z_k (z_k-1)\partial_k\right) c_{ip-j}\\
=& \left( m\sum_{k=1}^{2g-1}(z_k-1) + m-j +1 \right) c_{ip-j} + (m-g+j+1) c_{ip-(j+1)},
\end{align*}
Since $-m(2g-1)+m-j+1 = g-j$, the first assertion follows.

For $r=2$, applying Theorem~\ref{thm:Contiguity_Relations}({\romannumeral 1}) to the equality of the first assertion, we have
\begin{align*}
& \left(\sum_{k=1}^{2g-1} z_k^2 (z_k-1)\partial_k\right) c_{ip-(s-1)}\\
=& m\sum_{k=1}^{2g-1}z_k(z_k-1)c_{ip-(s-1)} +\left( m\sum_{k=1}^{2g-1}z_k +g -s \right) c_{ip-s} + (m-g+s+1) c_{ip-(s+1)}.
\end{align*}
The corollary follows inductively on $r$.
\end{proof}

\section{Proof of the multiplicity-one theorem}
We clarify what it means for an ideal to have multiplicity-one in an Artinian ring, and we prove the multiplicity-one theorem by means of the equalities involving partial derivatives of $c_{ip-j}$ established in Subsection 3.2.

\subsection{The Jacobian criterion}
Let $k$ be an algebraically closed field.
Let $R=k[x_1,\ldots,x_n]$.
Let $I$ be an ideal of $R$ such that $\quotientset[I]{R}$ is an Artinian ring (i.e., $I$ is zero-dimensional).
\begin{Def}
We say $I$ is of {\it multiplicity-one} if  $\dim_k (\quotientset[I]{R})_\mathfrak{P} = 1$ for any $\mathfrak{P}$ in $\mathrm{Spec}(\quotientset[I]{R})$.
\end{Def}
The following lemma is well-known.
See for example \cite[Lemma 5.2]{HY} for the proof.
\begin{Lem}\label{lem:JacobianCriterion}
$I=(f_1,\ldots,f_m)$ is muplicplicity one if and only if for any $\mathfrak{P}$ in $\mathrm{Spec}(\quotientset[I]{R})$ the Jacobian matrix $\left(\partial f_j/\partial x_i\right)$ at $\mathfrak{P}$ is of rank $n$.
\end{Lem}

\subsection{Proof of the multiplicity-one theorem}
In this subsection, we complete the proof of the the multiplicity-one theorem.
We first state a lemma that will be used repeatedly.
\begin{Lem}\label{lem:linear_elimination}
Let $V$ be a vector space over $K$ and let $\alpha_1,\alpha_2,\ldots,\alpha_r$ be distinct elements in $K$.
Let $m$ be a nonnegative integer with $m<r$ and $v_1,v_2,\ldots,v_r$ be elements of $V$ that satisfy
\begin{equation}\label{eq:many_relations}
\begin{aligned}
v_1 + v_2 \cdots + v_{r} &= 0,\\
\alpha_1 v_1 + \alpha_2 v_2 \cdots + \alpha_r v_{r} &= 0,\\
\alpha^2_1 v_1 + \alpha^2_2 v_2 \cdots + \alpha^2_r v_{r} &= 0,\\
\vdots\\
\alpha^m_1 v_1 + \alpha^m_2 v_2 \cdots + \alpha^m_r v_{r} &= 0.
\end{aligned}
\end{equation}
Then, the subspace spanned by $v_1,\ldots,v_r$ can be spanned by any $r-m-1$ vectors chosen from $v_1,\ldots,v_r$.
Specifically, if $k\lbrack1\rbrack,k\lbrack2\rbrack,\ldots,k\lbrack r-m-1\rbrack$ are distinct $r-m-1$ elements of $\{1,2,\ldots,r\}$, then for any $i \in \{1,\ldots,r\} \setminus \{k\lbrack1\rbrack,\ldots,k\lbrack r-m-1\rbrack\}$, $v_i$ lies in the span of $v_{k\lbrack1\rbrack},v_{k\lbrack2\rbrack},\ldots,v_{k\lbrack r-m-1\rbrack}$.
In other words, $\langle v_1,v_2,\ldots,v_r\rangle = \langle v_{k\lbrack1\rbrack},v_{k\lbrack2\rbrack},\ldots,v_{k\lbrack r-m-1\rbrack} \rangle$.
\end{Lem}
\begin{proof}
We prove the lemma by induction on $r$.
When $r=1$, \eqref{eq:many_relations} is nothing other than $v_1=0$, and thus the lemma holds.

Now assume that $r>1$ and that the lemma holds for $r-1$.
When $m=0$, the lemma is trivial.
Thus, we assume $0<m$.
By symmetry, it suffices to prove the lemma for the choice $(k\lbrack1\rbrack,k\lbrack2\rbrack,\ldots,k\lbrack r-m-1\rbrack)=(1,2,\ldots,r-m-1)$.
From the equalities \eqref{eq:many_relations}, we eliminate $v_r$, by multiplying with a suitable power of $\alpha_r$ and subtracting, to obtain
\begin{equation*}
\begin{aligned}
(\alpha_1-\alpha_r)v_1 + \cdots + (\alpha_{r-1}-\alpha_r)v_{r-1} &= 0,\\
\alpha_1(\alpha_1-\alpha_r) v_1 + \cdots + \alpha_{r-1}(\alpha_{r-1}-\alpha_r) v_{r-1} &= 0,\\
\vdots\\
\alpha^{m-1}_1(\alpha_1-\alpha_r) v_1 + \cdots + \alpha^{m-1}_{r-1}(\alpha_{r-1}-\alpha_r) v_{r-1} &= 0.
\end{aligned}
\end{equation*}
Our inductive hypothesis for $(\alpha_1-\alpha_r)v_1,\ldots,(\alpha_{r-1}-\alpha_r)v_{r-1}$ implies that for each $i \in \{r-m,\ldots,r-1\}$, $(\alpha_i-\alpha_r)v_i$ lies in the span of $(\alpha_1-\alpha_r)v_1,\ldots,(\alpha_1-\alpha_{r-m-1})v_{r-m-1}$. Since $\alpha_1,\alpha_2,\ldots,\alpha_r$ are distinct, $\alpha_i-\alpha_r$ is a unit in $K$.
Hence $v_i$ lies in the span of $v_1,\ldots,v_{r-m-1}$.
Eliminating $v_{r-1}$ instead of $v_r$ from \eqref{eq:many_relations}, we deduce by the same argument that $v_r$ also lies in the span of $v_1,\ldots,v_{r-m-1}$, completing the proof.
\end{proof}

Let $g$ be a positive integer and $K$ a field of characteristic $p>2$ with $g \le (p-1)/2$ (Remark~\ref{rem:restriction_on_primes}).
Let $C_g(\lambda)$ be the curve $y^2 = f(x)\coloneqq x(x-1)(x-\lambda_1)(x-\lambda_2)\cdots(x-\lambda_{2g-1})$ for $\lambda_1,\ldots,\lambda_{2g-1} \in K$.

We prove Theorem~\ref{thm:SspImpliesNonsing} and Theorem~\ref{thm:MultiplicityOne} by induction on genus $g$.
We have these results for $g=1$ by Igusa~\cite{Igusa} and $g=2$ by Harashita--Yamamoto~\cite{HY}.
We assume that $g>1$ and that Theorem~\ref{thm:SspImpliesNonsing} and Theorem~\ref{thm:MultiplicityOne} hold for genus less than $g$.
We first prove Theorem~\ref{thm:SspImpliesNonsing} for genus $g$, and then prove Theorem~\ref{thm:MultiplicityOne} for genus $g$ by using Theorem~\ref{thm:SspImpliesNonsing}.
By Lemma~\ref{lem:JacobianCriterion}, to prove Theorem~\ref{thm:MultiplicityOne} it suffices to show that the Jacobian matrix is full rank at any superspecial point, which we prove at the end of the section.

We introduce the notation used in this subsection.
\begin{Notation}\label{notation:JacobianMatrix}
When $\lambda\coloneqq(\lambda_1,\ldots,\lambda_{2g-1})$ is a point of $V(c_{ip-j}\,|\, 1\le i,j\le g)$, we denote the Jacobian matrix of the Cartier--Manin matrix at $\lambda$ by $J(\lambda)$.
Arranging rows and columns, we let $J(\lambda)$ be
\begin{equation*}
\begin{pNiceMatrix}
   \partial_1 c_{p-1}(\lambda) & \dots & \partial_1 c_{gp-1}(\lambda) & \dots &  \partial_1 c_{p-g}(\lambda) & \dots & \partial_1 c_{gp-g}(\lambda)\\
  \partial_2 c_{p-1}(\lambda) & \dots & \partial_2 c_{gp-1}(\lambda) & \dots &  \partial_2 c_{p-g}(\lambda) & \dots & \partial_2 c_{gp-g}(\lambda)\\
  \vdots & \ddots & \vdots & \cdots & \vdots & \ddots & \vdots\\
  \partial_{2g-1} c_{p-1}(\lambda) & \dots & \partial_{2g-1} c_{gp-1}(\lambda) & \dots &  \partial_{2g-1} c_{p-g}(\lambda) & \dots & \partial_{2g-1} c_{gp-g}(\lambda)\\
\CodeAfter
\UnderBrace[yshift=1.5mm,shorten]{4-1}{4-3}{j=1}%
\UnderBrace[yshift=1.5mm,shorten]{4-5}{4-last}{j=g}%
\end{pNiceMatrix}.
\end{equation*}

\vspace{1em}
\noindent We write its column vectors as $J(\lambda)=
\begin{pmatrix}
v_{1,1} & \dots & v_{g,1} & \dots & v_{1,g} & \dots & v_{g,g}
\end{pmatrix}$
and the $k$-th row vector of $J(\lambda)$ as $w_k$ for $k=1,\ldots,2g-1$, which is a $1 \times g^2$ matrix.
We denote the $j$-th part of $w_k$ by $w_k^{(j)}$.
Namely,
\[
w_k^{(j)} = 
\begin{pmatrix}
\partial_k c_{p-j}(\lambda) & \partial_k c_{2p-j}(\lambda) & \dots & \partial_k c_{gp-j}(\lambda)
\end{pmatrix}
\]
for $j=1,\ldots,g$, and then $w_k = 
\begin{pmatrix}
w_k^{(1)} & w_k^{(2)} & \dots & w_k^{(g)}
\end{pmatrix}$
for $k=1,\ldots,2g-1$.
\end{Notation}

Let $\lambda=(\lambda_1,\ldots,\lambda_{2g-1})$ be a point of $V(c_{ip-j}\,|\, 1\le i,j\le g)$.
In our induction setting, the following proposition corresponds to Theorem~\ref{thm:SspImpliesNonsing}.

\begin{Prop}\label{prop:SspImpliesNonsing_under_induction}
Assume that $C_g(\lambda)$ is superspecial; more precisely its Cartier--Manin matrix
\[
\begin{pmatrix} 
  c_{p-1} & c_{p-2} & \dots  & c_{p-g} \\
  c_{2p-1} & c_{2p-2} & \dots  & c_{2p-g} \\
  \vdots & \vdots & \ddots & \vdots \\
  c_{gp-1} & c_{gp-2} & \dots  & c_{gp-g}
\end{pmatrix}
\]
is zero when evaluated at $\lambda=(\lambda_1,\ldots,\lambda_{2g-1})$ and that Theorem~\ref{thm:MultiplicityOne} for genus $g-1$ holds.
Then we have $\lambda_k \ne 0, 1$ and $\lambda_k \ne \lambda_\ell$ if $k\ne \ell$ for $k,\ell\in\{1,\ldots,2g-1\}$.
\end{Prop}

When $g=1$, it is known, as we mentioned in Section 1, that $c_{p-1} = (-1)^{(p-1)/2} H_p(z_1)$, where $\displaystyle H_p(z_1)\coloneqq \sum_{i=0}^{(p-1)/2}{\binom{(p-1)/2}{i}}^2z_1^i.$
Since $H_p(0) = 1$ and $\displaystyle H_p(1) = \sum_{i=0}^{(p-1)/2}{\binom{(p-1)/2}{i}}^2 = \binom{p-1}{(p-1)/2} \equiv (-1)^\frac{p-1}{2} \pmod{p}$ are both nonzero, it follows that $\lambda_1 \ne 0,1$, and thus the proposition holds.

\begin{proof}[Proof of Proposition \ref{prop:SspImpliesNonsing_under_induction}]
Set $m\coloneqq(p-1)/2$.
As stated in the beginning of this subsection, we assume $g>1$ and that the proposition holds for genus less that $g$.
As in Definition~\ref{def:CMMatrix}, let $f(x)\coloneqq x(x-1)(x-z_1)(x-z_2)\cdots(x-z_{2g-1})$, and we denote the $x^k$-coefficient of $f(x)^m$ by $c_k$.

Because the choice of labeling among the branch points $0,1,\lambda_1,\ldots,\lambda_{2g-1}$ is arbitrary up to projective automorphism, any potential coincidence among the branch points can be moved, by an appropriate M{\"o}bius transformation, to the form $\lambda_{2g-1}=0$.
Under this transformation, the condition that the Cartier--Manin matrix vanishes is preserved (\cite[Proposition 2.2]{Yui}).
Therefore, we may assume that $\lambda_{2g-1}=0$ and show that this assumption leads to a contradiction.
Since $\lambda_{2g-1}=0$, we consider $f(x)=x^2(x-1)(x-z_1)(x-z_2)\cdots(x-z_{2g-2})$.
From $f(x)^m=x^{p-1}\left((x-1)(x-z_1)(x-z_2)\cdots(x-z_{2g-2})\right)^m$, we see that $c_{p-2},c_{p-3},\ldots,c_{p-g}$ are zero (as polynomials in $\bm{z}$).
Moreover, since $\lambda$ is a solution to $c_{p-1} = (-1)^{(2g-1)m}\left(z_1 z_2\cdots z_{2g-2}\right)^m$, we have $\lambda_1 \lambda_2 \cdots \lambda_{2g-2} = 0$, and thus we may assume $\lambda_{2g-2} = 0$ without loss of generality.
Hence, we consider $f(x)=x^3(x-1)(x-z_1)\cdots(x-z_{2g-3})$, in which case the Cartier--Manin matrix is
\[
\begin{pmatrix} 
  0 & 0 & \dots  & 0 \\
  c_{2p-1} & c_{2p-2} & \dots  & c_{2p-g} \\
  \vdots & \vdots & \ddots & \vdots \\
  c_{gp-1} & c_{gp-2} & \dots  & c_{gp-g}
\end{pmatrix}.
\]

Let $g(x)\coloneqq x(x-1)(x-z_1)\cdots(x-z_{2g-3})$, and $\delta_k$ be the $x^k$-coefficient of $g(x)^m$.
Then, the Cartier--Manin matrix of $y^2 = g(x)$ is $\widetilde{M}\coloneqq{ \Bigl( \delta_{ip-j} \Bigr) }_{1\le i,j\le g-1}$.
From $f(x)^m=x^{p-1}g(x)^m$, we see that $c_{ip-j}=\delta_{(i-1)p-(j-1)}$ for $i,j\in\{2,\ldots,g\}$.
By our inductive hypothesis, we have $\lambda_k \ne 0, 1$ and $\lambda_k \ne \lambda_\ell$ for $k,\ell\in\{1,\ldots,2g-3\}$ with $k\ne \ell$.
Now we turn to $\widetilde{M}$ and its Jacobian matrix at $\widetilde{\lambda}\coloneqq(\lambda_1,\ldots,\lambda_{2g-3})$ is
\begin{equation}\label{eq:Jacobian_matrix_minus_one}
\begin{pmatrix}
   \partial_1 \delta_{p-1}(\widetilde{\lambda}) & \dots & \partial_1 \delta_{(g-1)p-1}(\widetilde{\lambda}) & \dots & \partial_1 \delta_{(g-1)p-(g-1)}(\widetilde{\lambda})\\
  \partial_2 \delta_{p-1}(\widetilde{\lambda}) & \dots &  \partial_2 \delta_{(g-1)p-1}(\widetilde{\lambda}) & \dots & \partial_2 \delta_{(g-1)p-(g-1)}(\widetilde{\lambda})\\
  \vdots & \ddots & \vdots & \ddots & \vdots\\
  \partial_{2g-3} \delta_{p-1}(\widetilde{\lambda}) & \dots & \partial_{2g-3} \delta_{(g-1)p-1}(\widetilde{\lambda}) & \dots & \partial_{2g-3} \delta_{(g-1)p-(g-1)}(\widetilde{\lambda})
\end{pmatrix},
\end{equation}
which equals a submatrix of $J(\lambda)$:
\[
\begin{pmatrix}
   \partial_1 c_{2p-2}(\lambda) & \dots & \partial_1 c_{gp-2}(\lambda) & \dots & \partial_1 c_{gp-g}(\lambda)\\
  \partial_2 c_{2p-2}(\lambda) & \dots &  \partial_2 c_{gp-2}(\lambda) & \dots & \partial_2 c_{gp-g}(\lambda)\\
  \vdots & \ddots & \vdots & \ddots & \vdots\\
  \partial_{2g-3} c_{2p-2}(\lambda) & \dots & \partial_{2g-3} c_{gp-2}(\lambda) & \dots & \partial_{2g-3} c_{gp-g}(\lambda)
\end{pmatrix}.
\]
We denote the $k$-th row vector of \eqref{eq:Jacobian_matrix_minus_one} by $\widetilde{w}_k$ and the $j$-th part of $\widetilde{w}_k$ by $\widetilde{w}_k^{(j)}$ for $k=1,\ldots,2g-3$ (as in Notation~\ref{notation:JacobianMatrix}):
\[
\widetilde{w}_k^{(j)} =
\begin{pmatrix}
\partial_k \delta_{p-j}(\widetilde{\lambda}) & \partial_k \delta_{2p-j}(\widetilde{\lambda}) & \dots & \partial_k \delta_{(g-1)p-j}(\widetilde{\lambda})
\end{pmatrix}.
\]
Note that $w_k^{(j+1)} = \left( 0, \widetilde{w}_k^{(j)} \right)$ for $k \in \{1,\ldots,2g-3\}$ and $j \in \{1,\ldots,g-1\}$.

By Corollary~\ref{cor:Relations} for $\widetilde{M}$, we obtain
\begin{equation}\label{eq:relations_of_minus_one}
\begin{aligned}
\lambda_1(\lambda_1-1) \widetilde{w}_1^{(1)} + \cdots + \lambda_{2g-3}(\lambda_{2g-3}-1) \widetilde{w}_{2g-3}^{(1)} &= 0,\\
\lambda_1^2(\lambda_1-1) \widetilde{w}_1^{(1)} + \cdots + \lambda_{2g-3}^2(\lambda_{2g-3}-1) \widetilde{w}_{2g-3}^{(1)} &= 0,\\
\vdots\\
\lambda_1^{g-2}(\lambda_1-1) \widetilde{w}_1^{(1)} + \cdots + \lambda_{2g-3}^{g-2}(\lambda_{2g-3}-1) \widetilde{w}_{2g-3}^{(1)} &= 0.
\end{aligned}
\end{equation}
Moreover, by Theorem~\ref{thm:Contiguity_Relations} ({\romannumeral 1}) for $j=1,k=2g-2$, we have $\partial_{2g-2}c_{ip-2}(\lambda)=\lambda_{2g-2}\partial_{2g-2}c_{ip-1}(\lambda)-ec_{ip-1}(\lambda)=0$ for any $i$.
Similarly, $\partial_{2g-1}c_{ip-2}(\lambda)=0$ for any $i$.
Hence $w_{2g-2}^{(2)}=w_{2g-1}^{(2)}=0$.
Since $(\lambda_1-1)w_1^{(2)}+\cdots+(\lambda_{2g-1}-1)w_{2g-1}^{(2)}=0$ by Theorem~\ref{thm:Contiguity_Relations} ({\romannumeral 2}), we have
\begin{equation}\label{eq:add_relation_of_minus_one}
(\lambda_1-1) \widetilde{w}_1^{(1)} + \cdots + (\lambda_{2g-3}-1) \widetilde{w}_{2g-3}^{(1)} = 0.
\end{equation}
From the equalities \eqref{eq:add_relation_of_minus_one} and \eqref{eq:relations_of_minus_one}, and the hypothesis that $0,1,\lambda_1,\ldots,\lambda_{2g-3}$ are distinct, Lemma~\ref{lem:linear_elimination} applied with $v_1=(\lambda_1-1) \widetilde{w}_1^{(1)},\ldots,v_r=(\lambda_{2g-3}-1) \widetilde{w}_{2g-3}^{(1)}$ and with $m=g-2,r=2g-3$, implies that $(\lambda_{g-1}-1) \widetilde{w}_{g-1}^{(1)}$ is a linear combination of $(\lambda_1-1) \widetilde{w}_1^{(1)},\ldots,(\lambda_{g-2}-1) \widetilde{w}_{g-2}^{(1)}$.
Therefore, $\widetilde{w}_{g-1}^{(1)}$ is a linear combination of $\widetilde{w}_1^{(1)},\ldots,\widetilde{w}_{g-2}^{(1)}$.
On the other hand, Lemma~\ref{lem:rank_of_submatrices} applied to the case $y^2=g(x)$ of genus $g-1$ implies that $\widetilde{w}_1^{(1)},\ldots,\widetilde{w}_{g-2}^{(1)},\widetilde{w}_{g-1}^{(1)}$ are linearly independent, which leads to a contradiction, completing the proof.
\end{proof}

\begin{Lem}\label{lem:linear_independence}
$v_{1,j},v_{2,j},\ldots,v_{g,j}$ are linearly independent for $j = 1,\ldots,g$.
\end{Lem}
\begin{proof}
We first claim that $c_{p-j},c_{2p-j},\ldots,c_{gp-j}$ are linearly independent over $K$.
For any $i \in \{1,\ldots,g-1\}$, the difference between the lowest degree among the monomials of $c_{ip-j}$ and the highest degree among those of $c_{(i+1)p-j}$ is $\left((p-1)g-ip+j\right) - \left((p-1)(g+1/2)-(i+1)p+j\right) = (p+1)/2 > 0$.
Hence $\supp[1][j],\ldots,\supp[g][j]$ are disjoint, and the calim follows.

Suppose, toward a contradiction, that $v_{1,j},v_{2,j},\ldots,v_{g,j}$ are linearly dependent.
Then there is a nontrivial linear combination $\alpha_1 v_{1,j} + \alpha_2 v_{2,j} + \cdots + \alpha_g v_{g,j} = 0$ with $\alpha_1,\ldots,\alpha_g$ not all zero.
There exists $\ell \in \{1,\ldots,g\}$ such that $\alpha_\ell \ne 0$.
Write
\begin{equation}\label{eq:linear_independence_contradiction}
v_{\ell,j} = \widetilde{\alpha}_1 v_{1,j} +\cdots + \widetilde{\alpha}_{\ell-1} v_{\ell-1,j} + \widetilde{\alpha}_{\ell+1} v_{\ell+1,j} +\cdots +\widetilde{\alpha}_g v_{g,j},
\end{equation}
where $\widetilde{\alpha}_i \coloneqq -\alpha^{-1}_\ell \alpha_i$ for any $i \in \{1,\ldots,g\} \setminus \{\ell\}$.
We consider, for any $i \in \{1,\ldots,g\}$, the expansion of $c_{ip-j}$
at $\lambda\coloneqq(\lambda_1,\ldots,\lambda_{2g-1})$:
As the $z_\ell$-degree ($\ell = 1,\ldots,2g-1$) of $c_{ip-j}$  is at most $(p-1)/2$, 
we can write $c_{ip-j}$ as
\[
\sum_{k_1,\ldots,k_{2g-1}} \frac{1}{k_1!\cdots k_{2g-1}!}\frac{\partial^{k_1+\dots+k_{2g-1}}c_{ip-j}}{\partial z_1^{k_1}\dots\partial z_{2g-1}^{k_{2g-1}}} (\lambda)
(z_1-\lambda_1)^{k_1} \cdots (z_{2g-1}-\lambda_{2g-1})^{k_{2g-1}}.
\]
From \eqref{eq:linear_independence_contradiction}, the expansion at $\lambda$ of $c_{\ell p-j}$ is a linear combination of those of $c_{ip-j}$ for $i \in \{1,\ldots,g\} \setminus \{\ell\}$ up to degree one.
Since $c_{p-j},\ldots,c_{gp-j}$ satisfy the same partial differential equations in Theorem~\ref{thm:PDEinPositiveChar}, an inductive argument shows that $c_{\ell p-j} = \widetilde{\alpha}_1 c_{p-j} +\cdots + \widetilde{\alpha}_{\ell-1} c_{(\ell-1)p-j} + \widetilde{\alpha}_{\ell+1} c_{(\ell+1)p-j} +\cdots +\widetilde{\alpha}_g c_{gp-j}$, contradicting the first paragraph.
\end{proof}

\begin{Lem}\label{lem:rank_of_submatrices}
Any $g\times g$ submatrix of $\begin{pmatrix} v_{1,1} & v_{2,1} & \dots & v_{g,1} \end{pmatrix}$ is of rank $g$.
\end{Lem}
\begin{proof}
We choose distinct $g$ elements $k\lbrack1\rbrack,k\lbrack2\rbrack,\ldots,k\lbrack g\rbrack$ from $\{1,2,\ldots,2g-1\}$ arbitrarily, and prove the lemma for the submatrix formed by $w_{k\lbrack1\rbrack}^{(1)},\ldots,w_{k\lbrack g\rbrack}^{(1)}$.
We claim that the row space of $\begin{pmatrix} v_{1,1} & v_{2,1} & \dots & v_{g,1} \end{pmatrix}$, which is spanned by $w_1^{(1)},w_2^{(1)},\ldots,w_{2g-1}^{(1)}$, can be spanned by $w_{k\lbrack1\rbrack}^{(1)},\ldots,w_{k\lbrack g\rbrack}^{(1)}$.
Once we prove the claim, the lemma follows by Lemma~\ref{lem:linear_independence}.
By Corollary~\ref{cor:Relations}, we obtain
\begin{equation}\label{eq:ample_relations}
\begin{aligned}
\lambda_1(\lambda_1-1) w_1^{(1)} + \cdots + \lambda_{2g-1}(\lambda_{2g-1}-1) w_{2g-1}^{(1)} &= 0,\\
\lambda_1^2(\lambda_1-1) w_1^{(1)} + \cdots + \lambda_{2g-1}^2(\lambda_{2g-1}-1) w_{2g-1}^{(1)} &= 0,\\
\vdots\\
\lambda_1^{g-1}(\lambda_1-1) w_1^{(1)} + \cdots + \lambda_{2g-1}^{g-1}(\lambda_{2g-1}-1) w_{2g-1}^{(1)} &= 0.
\end{aligned}
\end{equation}
From these equalities and Proposition~\ref{prop:SspImpliesNonsing_under_induction}, Lemma~\ref{lem:linear_elimination} applied with $v_1=\lambda_1(\lambda_1-1) w_1^{(1)},\ldots,v_r=\lambda_{2g-1}(\lambda_{2g-1}-1) w_{2g-1}^{(1)}$ and with $m=g-2,r=2g-1$, implies that the span of $\lambda_1(\lambda_1-1)w_1^{(1)},\lambda_2(\lambda_2-1)w_2^{(1)},\ldots,\lambda_{2g-1}(\lambda_{2g-1}-1)w_{2g-1}^{(1)}$ equals the span of $\lambda_{k\lbrack1\rbrack}(\lambda_{k\lbrack1\rbrack}-1)w_{k\lbrack1\rbrack}^{(1)},\lambda_{k\lbrack2\rbrack}(\lambda_{k\lbrack2\rbrack}-1)w_{k\lbrack2\rbrack}^{(1)},\ldots,\lambda_{k\lbrack g\rbrack}(\lambda_{k\lbrack g\rbrack}-1)w_{k\lbrack g\rbrack}^{(1)}$.
Since $\lambda_1(\lambda_1-1),\ldots,\lambda_{2g-1}(\lambda_{2g-1}-1)$ are units in $K$, these spans coincide:
\[
\langle w_1^{(1)},w_2^{(1)},\ldots,w_{2g-1}^{(1)} \rangle = \langle w_{k\lbrack1\rbrack}^{(1)},\ldots,w_{k\lbrack g\rbrack}^{(1)} \rangle,
\]
and the claim is proved.
\end{proof}

\begin{Thm}\label{thm:rank_maximum}
The rank of $J(\lambda)$ is $2g-1$.
\end{Thm}
\begin{proof}
We prove that the row rank of $J(\lambda)$ is $2g-1$.
For $\ell \in \{1,\ldots,g-1\}$, define
\begin{equation*}
\widehat{w}_\ell \coloneqq \lambda^\ell_1(\lambda_1 -1)w_1 + \cdots + \lambda^\ell_{2g-1}(\lambda_{2g-1} -1)w_{2g-1},
\end{equation*}
which is in the row space of $J(\lambda)$ (see Notation~\ref{notation:JacobianMatrix}).
We show that $2g-1$ row vectors
\[
w_1,w_2,\ldots,w_g,\widehat{w}_1,\widehat{w}_2,\ldots,\widehat{w}_{g-1}
\]
are linearly independent.
Suppose that we have an arbitrary linear relation
\begin{equation}\label{eq:linear_independence_to_be_shown}
\alpha_1 w_1 +\cdots+ \alpha_g w_g + \beta_1 \widehat{w}_1 +\cdots+ \beta_{g-1} \widehat{w}_{g-1} = 0,
\end{equation}
where $\alpha_1,\ldots,\alpha_g,\beta_1,\ldots,\beta_{g-1} \in K$.
By Corollary~\ref{cor:Relations}, $\widehat{w}_\ell\,(\ell=1,\ldots,g-1)$ looks as follows:
\begin{equation}\label{eq:hatvector}
\widehat{w}_\ell =
\begin{pNiceMatrix}
   \mathbf{0} & \mathbf{0} & \cdots & \mathbf{0} & \bigstar & \ast & \cdots & \ast\\
   j=1        & 2          & \cdots & g-\ell     & g-\ell+1 & g-\ell+2 &   \cdots & g
\end{pNiceMatrix},
\end{equation}
where each symbol $\mathbf{0},\bigstar,\ast$ is comprised of $g$ entries, and the symbol $\bigstar$ at the position $j=g-\ell+1$ is $\lambda^\ell_1(\lambda_1-1)w^{(g-\ell+1)}_{1} +\cdots +\lambda^\ell_{2g-1}(\lambda_{2g-1}-1)w^{(g-\ell+1)}_{2g-1}$.
Since, for any $\ell \in \{1,\ldots,g-1\}$, the first $g$ entries in the position of $j=1$ of $\widehat{w}_\ell$ are zero, from \eqref{eq:linear_independence_to_be_shown} we have $\alpha_1 w^{(1)}_1 +\cdots+ \alpha_g w^{(1)}_g = 0$.
By Lemma~\ref{lem:rank_of_submatrices}, we obtain $\alpha_1 = \cdots = \alpha_g = 0$, and \eqref{eq:linear_independence_to_be_shown} becomes $\beta_1 \widehat{w}_1 +\cdots+ \beta_{g-1} \widehat{w}_{g-1} = 0$.

Here, we claim that for any $\ell \in \{1,\ldots,g-1\}$, $\bigstar$ in \eqref{eq:hatvector} is nonzero.
This will lead to the linear independence of $\widehat{w}_1,\widehat{w}_2,\ldots,\widehat{w}_{g-1}$.
To verify this claim, suppose to the contrary that $\bigstar$ is zero for some $\ell$:
\begin{equation}\label{eq:asumption_toward_contradiction}
\lambda^\ell_1(\lambda_1-1)w^{(g-\ell+1)}_{1} +\cdots +\lambda^\ell_{2g-1}(\lambda_{2g-1}-1)w^{(g-\ell+1)}_{2g-1} = 0.
\end{equation}
By Theorem~\ref{thm:Contiguity_Relations} ({\romannumeral 1}), we obtain the relations $w^{(j)}_k = \lambda_k w^{(j-1)}_k$ for any $k\in\{1,\ldots,2g-1\}$ and for any $j\in\{2,3,\ldots,g\}$.
From these relations, we rewrite \eqref{eq:asumption_toward_contradiction} as
\[
\lambda_1^{g}(\lambda_1-1) w_1^{(1)} + \cdots + \lambda_{2g-1}^{g}(\lambda_{2g-1}-1) w_{2g-1}^{(1)} = 0.
\]
From this equality together with the equalities \eqref{eq:ample_relations}, and Proposition~\ref{prop:SspImpliesNonsing_under_induction}, we apply Lemma~\ref{lem:linear_elimination} with $v_1=\lambda_1(\lambda_1-1) w_1^{(1)},\ldots,v_r=\lambda_{2g-1}(\lambda_{2g-1}-1) w_{2g-1}^{(1)}$ and with $m=g-1,r=2g-1$, to deduce that $w_{g}^{(1)}$ is a linear combination of $w_1^{(1)},w_2^{(1)},\ldots,w_{g-1}^{(1)}$.
This contradicts the linear independence of $w_1^{(1)},w_2^{(1)},\ldots,w_{g}^{(1)}$ (Lemma~\ref{lem:rank_of_submatrices}), thereby establishing our claim.

Thus, $\widehat{w}_1,\widehat{w}_2,\ldots,\widehat{w}_{g-1}$ are linearly independent and hence $\beta_1=\cdots=\beta_{g-1}=0$, and we have completed the proof.
\end{proof}

\begin{Rem}
In the proof above, choosing any distinct $g$ elements $k\lbrack1\rbrack,k\lbrack2\rbrack,\ldots,k\lbrack g\rbrack$ from $\{1,2,\ldots,2g-1\}$, we may replace $w_1,w_2,\ldots,w_g$ with $w_{k\lbrack1\rbrack},w_{k\lbrack2\rbrack},\ldots,w_{k\lbrack g\rbrack}$.
Namely, $w_{k\lbrack1\rbrack},w_{k\lbrack2\rbrack},\ldots,w_{k\lbrack g\rbrack},\widehat{w}_1,\widehat{w}_2,\ldots,\widehat{w}_{g-1}$ are linearly independent.
\end{Rem}


\begin{thebibliography}{99}
\bibitem{Deuring}
M.\ Deuring:
{\it Die Typen der Multiplikatorenringe elliptischer Funktionenk\"orper},
Abh.\ Math.\ Sem.\ Univ.\ Hamburg, {\bf 14} (1941), no.\ 1, 197--272. 

\bibitem{Ekedahl}
T.\ Ekedahl:
\textit{On supersingular curves and abelian varieties},
Mathematica Scandinavica,
{\bf 60} (1987), 151--178.

\bibitem{HY}
S.\ Harashita and Y.\ Yamamoto:
\textit{The multiplicity-one theorem for the superspeciality of curves of genus two,}
Finite Fields and Their Applications,
{\bf 110} (2026), 102738.

\bibitem{IKO}
T.\ Ibukiyama, T.\ Katsura and F.\ Oort:
{\it Supersingular curves of genus two and class numbers}, 
Compositio Math., {\bf 57} (1986), no.\ 2, 127--152. 

\bibitem{Igusa}
J.\ Igusa:
\textit{Class\ number\ of\ a\ definite\ quaternion\ with\ prime\ discriminant,}
1958, Proc. Nat. Acad. Sci. U.S.A.

\bibitem{KHS}
M.\ Kudo, S.\ Harashita and H.\ Senda:
\textit{The existence of supersingular curves of genus 4 in arbitrary characteristic,}
Research in Number Theory {\bf 6}, Article number: 44 (2020).

\bibitem{Manin}
Y.\ I.\ Manin:
{\it The theory of commutative formal groups over fields of finite characteristic,} 
Russ. Math. Surv. {\bf 18}, No. 6, 1-83 (1963); translation from Usp. Mat. Nauk {\bf 18}, No. 6(114), 3-90 (1963).

\bibitem{Matsumoto}
K.\ Matsumoto: \textit{Appell and Lauricella Hypergeometric Functions,} In T.\ Koornwinder \& J.\ Stokman (Eds.), Encyclopedia of Special Functions: The Askey-Bateman Project (pp. 79-100) (2020). Cambridge: Cambridge University Press. doi:10.1017/9780511777165.004

\bibitem{Nygaard}
N.\ O.\ Nygaard:
\textit{Slopes of powers of Frobenius on crystalline cohomology,}
Ann. Sci. \'Ec. Norm. Sup\'er. (4) {\bf 14}, 369--401 (1981).

\bibitem{OH}
R.\ Ohashi and S.\ Harashita:
\textit{Differential\ forms\ on the\ curves\ associated\ to\
Appell-\hspace{0pt}Lauricella\ hypergeometric\ series and\ the\
Cartier\ operator\ on\ them,}
Yokohama Mathematical Journal {\bf 69} (2023), 1--32.

\bibitem{Silverman}
J.\ H.\ Silverman: \textit{The\ Arithmetic\ of\ Elliptic\ Curves,}
Second edition, Graduate Texts in Mathematics, \textbf{106}, 2009, Springer.

\bibitem{Yui}
N.\ Yui:
\textit{On the Jacobian varieties of hyperelliptic curves over fields of characteristic $p > 2$,}
J.\ Algebra {\bf 52}, 378--410 (1978).
\end{thebibliography}
\end{document}